\documentclass{article}
\usepackage{amssymb,amsmath,bm,geometry,graphics,graphicx,url,color}
\usepackage{amsfonts}
\usepackage{graphicx}
\usepackage{multirow}
\usepackage{amscd}
\usepackage{enumerate}
\usepackage{mathrsfs}
\usepackage{xypic}
\usepackage{stmaryrd}
\usepackage{fancybox}
\usepackage{exscale}
\usepackage{enumitem,array}%
\usepackage{graphicx}

\def\pdt2{\partial_t^2}
\def\pdx2{\partial_x^2}

\newcommand{\normmm}[1]{{\left\vert\kern-0.25ex\left\vert\kern-0.25ex\left\vert #1
    \right\vert\kern-0.25ex\right\vert\kern-0.25ex\right\vert}}

\newcommand{\abs}[1]{\left\vert#1\right\vert}

\def\RR{{\mathbb{R}}}

\newtheorem{mytheo}{Theorem}[section]

\newtheorem{mydef}[mytheo]{Definition}
\newtheorem{cor}[mytheo]{Corollary}
\newtheorem{lem}[mytheo]{Lemma}

\newtheorem{rem}[mytheo]{Remark}

\def\no{\noindent}

\title{Efficient energy-preserving methods for charged-particle dynamics}

\author{Ting Li\,\footnote{School of Mathematical Sciences, Qufu Normal
University, Qufu  273165,  P.R.China. E-mail:~{\tt
1009587520@qq.com}} \and Bin Wang\thanks{School of Mathematical
Sciences, Qufu Normal University, Qufu  273165, P.R.China;
Mathematisches Institut, University of T\"{u}bingen, Auf der
Morgenstelle 10, 72076 T\"{u}bingen, Germany. E-mail:~{\tt
wang@na.uni-tuebingen.de}} }

\begin{document}
\maketitle
\begin{abstract}
In this paper, energy-preserving methods are formulated and studied
for solving  charged-particle dynamics. We first formulate the
scheme of energy-preserving methods and analyze its basic properties
including  algebraic order and symmetry. Then it is shown that these
novel methods can exactly preserve the energy of charged-particle
dynamics. Moreover, the long time momentum conservation is studied
along such energy-preserving methods. A numerical experiment is
carried out to illustrate the notable superiority of the new methods
in comparison with the popular Boris method in the literature.
\medskip

\no{Keywords:}   Charged particle dynamics;  Energy-preserving
 methods; Long-time conservtion

\medskip
\no{MSC (2000):} 65P10,  65L05
\end{abstract}

%=============================================================================================
\section{Introduction}
A large amount of work in the literature has been devoted to
studying the  following charged-particle dynamics (see, e.g.
\cite{Arnold97,Boris,Cary2009,Ellison2015,
Hairer2017,Lubich2017,He2015,Qin2013,Tao2016, Zhang2016})
\begin{equation}\label{charged-particle}
\ddot{x}=\dot{x} \times  B(x) +F(x),\ \ \ \ \ x(t_{0})=x^{0},\ \
\dot{x}(t_{0})=\dot{x}^{0}
\end{equation}
where  $x(t)\in \RR^3$ represents the  position of a particle moving
in an electro-magnetic field, $B(x)$ is a magnetic field
 which is defined as $B(x)=\nabla_{x} \times A(x)$ with
the vector potential $ A(x) \in \RR^3 $, and $F(x)$ is the negative
gradient of the scalar potential $U(x)$. We define $v=\dot{x}$ and
then the energy of the dynamics is given by
  \begin{equation}\label{E(x,v)}
E(x,v)=\frac{1}{2}\abs{v}^2+U(x).
\end{equation}
It is well known that the solution of this system conserves the
energy  exactly, i.e.
$$E(x(t),v(t))\equiv E(x^0,\dot{x}^0)\ \ \ \textmd{for\ any } t. $$
It has been shown in \cite{Hairer2017} that if   the $ U(x) $ and $
A(x) $ have the following properties
\begin{equation}\label{M}
U(e^{\tau S}x)=U(x)  \ \ \textmd{and }\ \ e^{-\tau S}A(e^{\tau
S}x)=A(x)\ \ \ \ \ \textmd{for\ \  all\ \  real}\ \  \tau,
\end{equation}
where $ S $ is a skew-symmetric matrix, then the momentum
\begin{equation}
M(x,v)=(v+A(x))^{\intercal}Sx
\end{equation}
is conserved  along  the solution of the differential equation
\eqref{charged-particle}. This point can be proved by multiplying
\eqref{charged-particle} with $ Sx $ and the reader is referred to
\cite{Hairer2017} for details. It has also been noted in
\cite{Hairer2017}  that since the matrix $ S $ is skew-symmetric, we
have $ x^{\intercal}S\ddot{x}=-\frac{d}{dt}(x^{\intercal}S\dot{x}) $
and it follows from these properties that $ S\nabla U(x)=0 $ and $
x^{\intercal}S(\dot{x}\times
B(x))=-\frac{d}{dt}(x^{\intercal}SA(x)).$

In this paper, we denote the vector $B(x)$
 by $B(x)=(B_{1}(x),B_{2}(x),B_{3}(x))^{\intercal}$, where $B_{i}(x) \in \RR$ for $i=1,2,3.$  By the definition of the cross
 product, we obtain $ \dot{x} \times  B(x) =  \tilde{B}(x) \dot{x},$ where the skew-symmetric
matrix $\tilde{B}(x)$ is given by
$$\tilde{B}(x)=\left(
                   \begin{array}{ccc}
                     0 & B_{3}(x) & -B_{2}(x) \\
                     -B_{3}(x) & 0 & B_{1}(x) \\
                     B_{2}(x) & -B_{1}(x) & 0 \\
                   \end{array}
                 \right).
$$

In order to solve the charged-particle dynamics  effectively, many
kinds of useful methods have been studied and developed. Boris
method \cite{Boris} is a popular integrator and it was researched
further in \cite{Ellison2015, Hairer2017, Qin2013}. There are many
other kinds of methods which have been researched for solving
charged-particle dynamics, such as volume-preserving algorithms in
\cite{He2015}, symmetric multistep methods in \cite{Lubich2017}
 and symplectic or K-symplectic integrators in \cite{Tao2016, Zhang2016}.
Recently, the authors in \cite{wang2018-new2} proposed adapted
exponential integrators for solving charged-particle dynamics  and
analyzed its symplecticity.

On the other hand, energy-preserving methods are an important and
efficient kind of  methods which have been received much attention
in the past few years. The authors in \cite{E.Faou2004} constructed
energy-preserving $B$-series methods. The Average Vector Field (AVF)
method was presented in \cite{Celloni2014,Chen2001} and it was shown
in \cite{Quispel2008} that AVF method is also a $B$-series method.
In \cite{wang2012-1}, the authors proposed a new trigonometric
energy-preserving method. Various different kinds of
energy-preserving methods are proposed and analyzed, such as
discrete gradient methods (see, e.g. \cite{McLachlan2014, wang}),
the energy-preserving exponentially-fitted methods(see, e.g.
\cite{Miyatake2014, Miyatake2015}), time finite elements methods
(see, e.g. \cite{Betsch2000, Wu2016}), the Runge-Kutta-type
energy-preserving collocation methods(see, e.g. \cite{Cohen2011, E.
Hairer}) and Hamiltonian Boundary Value Methods ( see, e.g. \cite{L.
Brugnano}). We  refer to \cite{W2016, McLachlan1999,
Quispel2008,wang-2016,wu2017-JCAM, Wu2018} for more research work on
energy-preserving methods. However, it seems that energy-preserving
methods for solving charged-particle dynamics have not been
considered in the literature, which motivates this paper.

Based on these work, we  will formulate and research a novel
energy-preserving method for solving charged-particle dynamics
\eqref{charged-particle}. The rest of this paper is organized as
follows. In Section \ref{sec2}, we   present the scheme of the
method and analyze its algebraic order and symmetry. In Section
\ref{sec3}, it is shown  that the novel method can exactly preserve
the energy \eqref{E(x,v)} of charged-particle dynamics. The long
time near conservation of the momentum for this new method is
discussed in Section \ref{sec4}. Section \ref{sec6} reports a
numerical experiment to show the efficiency of the novel method.
Section \ref{sec7} is devoted to the conclusions of this paper.

\section{The scheme of the method and its basic properties}\label{sec2}
\subsection{Formulation of the method}

In order to drive effective methods for the system
\eqref{charged-particle}, we first present its exact solution by the
 variation-of-constants formula.
\begin{mytheo}
(See \cite{hairer2006}.) The exact solution of
\eqref{charged-particle} can be expressed as
\begin{equation}\label{exact}
\begin{array}
[c]{ll}
&x(t_{n}+h)=x(t_{n})+hv(t_{n})+h^{2}\int_{0}^1 (1-z)\hat{f}(t_{n}+hz)dz,\\
&v(t_{n}+h)=v(t_{n})+h \int_{0}^1 \hat{f}(t_{n}+hz)dz,
\end{array}
\end{equation}
where $ \hat{f}(t):=\tilde{B}(x(t))v(t) +F(x(t)) $.
\end{mytheo}

Based on the  variation-of-constants formula, we define  the
following method for the charged-particle dynamics
\eqref{charged-particle}.

\begin{mydef}
The energy-preserving method  for solving charged-particle dynamics
\eqref{charged-particle} is defined as
\begin{equation}\label{EP}
\begin{array}[c]{ll}
&x_{n+1}=x_{n}+hv_{n}+\frac{h^2}{2}\int_{0}^1
F(x_{n}+\tau(x_{n+1}-x_{n}))
d\tau+\frac{h^2}{2}\tilde{B}(\frac{x_{n+1}+x_{n}}{2})\frac{v_{n+1}+v_{n}}{2}, \\
&v_{n+1}=v_{n}+h\int_{0}^1 F(x_{n}+\tau(x_{n+1}-x_{n}))
d\tau+h\tilde{B}(\frac{x_{n+1}+x_{n}}{2})\frac{v_{n+1}+v_{n}}{2},
\end{array}
\end{equation}
where $h$ is a stepsize. We denote this method by EP.
\end{mydef}

\subsection{Algebraic order}\label{sec5}
\begin{mytheo}
Under the local assumptions $x(t_{n})=x_{n}$ and $v(t_{n})=v_{n}$,
the energy-preserving method \eqref{EP} is of order two, i.e.,
$$ x(t_{n+1})-x_{n+1}=\mathcal{O}(h^{3}),\ \ \ \ \ \ \ v(t_{n+1})-v_{n+1}=\mathcal{O}(h^{3}).$$
\end{mytheo}
\begin{proof}
By \eqref{exact} and \eqref{EP}, we compute
\begin{equation}\label{order}
\begin{array}
[c]{lll}
&x(t_{n+1})-x_{n+1}\\
=&(x(t_{n})+hv(t_{n})+h^{2}\int_{0}^{1}(1-z)\hat{f}(t_{n}+hz)dz)\\
&-(x_{n}+hv_{n}+\frac{h^2}{2}\int_{0}^{1}F(x_{n}+\tau(x_{n+1}-x_{n}))d\tau
+\frac{h^2}{2}\tilde{B}(\frac{x_{n+1}+x_{n}}{2})\frac{v_{n+1}+v_{n}}{2})\\
=&(x(t_{n})+hv(t_{n})+h^{2}\int_{0}^{1}(1-z)\tilde{B}(x(t_{n}+hz))v(t_{n}+hz)dz+h^{2}\int_{0}^{1}(1-z)F(x(t_{n}+hz))dz)\\
&-(x_{n}+hv_{n}+\frac{h^{2}}{2}\int_{0}^{1}F(x_{n}+\tau(x_{n+1}-x_{n}))d\tau+\frac{h^{2}}{2}\tilde{B}(\frac{x_{n+1}+x_{n}}{2})\frac{v_{n+1}+v_{n}}{2})\\
=&h^{2}(\int_{0}^{1}(1-z)\tilde{B}(x(t_{n}))v(t_{n})dz+\mathcal{O}(h)+\int_{0}^{1}(1-z)F(x(t_{n}))dz+\mathcal{O}(h))\\
&-(\frac{h^{2}}{2}\int_{0}^{1}F(x_{n}+\tau(x_{n+1}-x_{n}))d\tau+\frac{h^{2}}{2}\tilde{B}(\frac{x_{n+1}+x_{n}}{2})\frac{v_{n+1}+v_{n}}{2})\\
=&h^{2}\int_{0}^{1}(1-z)\tilde{B}(x_{n})v_{n}dz+h^{2}\int_{0}^{1}(1-z)F(x_{n})dz-\frac{h^{2}}{2}\tilde{B}(\frac{x_{n+1}+x_{n}}{2})\frac{v_{n+1}+v_{n}}{2}\\
&-\frac{h^{2}}{2}\int_{0}^{1}F(x_{n}+\tau(x_{n+1}-x_{n}))d\tau+\mathcal{O}(h^{3})\\
=&\frac{h^{2}}{2}(\tilde{B}(x_{n})v_{n}-\tilde{B}(\frac{x_{n+1}+x_{n}}{2})\frac{v_{n+1}+v_{n}}{2})+\frac{h^{2}}{2}(F(x_{n})-\int_{0}^{1}F(x_{n}+\tau(x_{n+1}+x_{n}))d\tau)
+\mathcal{O}(h^{3}).
\end{array}
\end{equation}
In the light of  $ x_{n+1}=x_{n}+\mathcal{O}(h)$ and $
v_{n+1}=v_{n}+\mathcal{O}(h), $  it is obtained that
$$\tilde{B}(\frac{x_{n+1}+x_{n}}{2})\frac{v_{n+1}+v_{n}}{2}=\tilde{B}(x_{n})v_{n}+\mathcal{O}(h).$$
On the other hand, we have
 $$ \int_{0}^{1}F(x_{n}+\tau(x_{n+1}+x_{n}))d\tau=F(x_{n})+\mathcal{O}(h).$$
Consequently, the formula \eqref{order} becomes
$$x(t_{n+1})-x_{n+1}=\mathcal{O}(h^{3}).$$
Similarly, we obtain $v(t_{n+1})-v_{n+1}=\mathcal{O}(h^{3}).$ The
proof is complete.
\end{proof}
\subsection{Symmetry of the method}
A numerical method denoted by $y_{n+1}=\Phi_h(y_n)$ is called to be
symmetric if exchanging $y_n\leftrightarrow
 y_{n+1}$ and $h\leftrightarrow -h$ does not change the scheme of the method  (see \cite{hairer2006}).
It has been pointed out in \cite{hairer2006} that
  symmetric methods  have excellent longtime behaviour  and they play
  an important
role in  geometric numerical integration.
\begin{mytheo}
The method \eqref{EP} is symmetric.
\end{mytheo}
\begin{proof}
Exchanging $x_{n}\leftrightarrow x_{n+1} , v_{n} \leftrightarrow
v_{n+1} $ and $h\leftrightarrow -h$ in \eqref{EP} yields
\begin{equation}\label{5}
\begin{array}[c]{ll}
&x_{n}=x_{n+1}-hv_{n+1}+\frac{h^2}{2}\int_{0}^1
F(x_{n+1}+\tau(x_{n}-x_{n+1}))
d\tau+\frac{h^2}{2}\tilde{B}(\frac{x_{n}+x_{n+1}}{2})\frac{v_{n}+v_{n+1}}{2}, \\
&v_{n}=v_{n+1}-h\int_{0}^1 F(x_{n+1}+\tau(x_{n}-x_{n+1}))
d\tau-h\tilde{B}(\frac{x_{n}+x_{n+1}}{2})\frac{v_{n}+v_{n+1}}{2}.
\end{array}
\end{equation}
From formula \eqref{5}, it follows that
\begin{equation}\label{6}
\begin{array}[c]{ll}
&x_{n+1}=x_{n}+hv_{n}+\frac{h^2}{2}\int_{0}^1
F(x_{n+1}+\tau(x_{n}-x_{n+1}))
d\tau+\frac{h^2}{2}\tilde{B}(\frac{x_{n+1}+x_{n}}{2})\frac{v_{n+1}+v_{n}}{2}, \\
&v_{n+1}=v_{n}+h\int_{0}^1 F(x_{n+1}+\tau(x_{n}-x_{n+1}))
d\tau+h\tilde{B}(\frac{x_{n+1}+x_{n}}{2})\frac{v_{n+1}+v_{n}}{2}.
\end{array}
\end{equation}
Letting $\sigma=1-\tau$ yields
\begin{equation}\nonumber
\begin{array}[c]{ll}
&\int_{0}^{1}F(x_{n+1}+\tau(x_{n}-x_{n+1}))d\tau=\int_{0}^{1}F(x_{n}+(1-\tau)(x_{n+1}-x_{n}))d\tau\\
&=-\int_{1}^{0}F(x_{n}+\sigma(x_{n+1}-x_{n}))d\sigma=\int_{0}^{1}F(x_{n}+\tau(x_{n+1}-x_{n}))d\tau.
\end{array}
\end{equation}
which shows that \eqref{6} is the same as \eqref{EP}. Therefore, the
method \eqref{EP} is symmetric.
\end{proof}
\section{Energy-preserving property}\label{sec3}
In this section, we show the energy-preserving property of the
method \eqref{EP}.
\begin{mytheo}\label{thm ep1}
The method \eqref{EP} preserves the energy $E$ in \eqref{E(x,v)}
exactly, i.e.,
$$E(x_{n+1},v_{n+1})=E(x_{n},v_{n}) \ \ \ \ \ for  \ \ \ \   n=0,1,\ldots $$
\end{mytheo}
\begin{proof}
In this paper, we denote
$\digamma:=\int_{0}^{1}F(x_{n}+\tau(x_{n+1}-x_{n})) d\tau $. We
compute
\begin{equation}\label{1}
E(x_{n+1},v_{n+1})=\frac{1}{2}v_{n+1}^{\intercal}v_{n+1}+U(x_{n+1}).
\end{equation}
Keeping the fact in mind that $\tilde{B}(x)$ is skew-symmetric and
inserting the second formula of \eqref{EP} into \eqref{1} yields
\begin{equation}\label{2}
\begin{array}[c]{llll}
&E(x_{n+1},v_{n+1})\\
=&\frac{1}{2}(v_{n}+h\digamma+h\tilde{B}(\frac{x_{n+1}+x_{n}}{2})\frac{v_{n+1}+v_{n}}{2})^{\intercal}(v_{n}+h\digamma
+h\tilde{B}(\frac{x_{n+1}+x_{n}}{2})\frac{v_{n+1}+v_{n}}{2})+U(x_{n+1})\\
=&\frac{1}{2}v_{n}^{\intercal}v_{n}+\frac{h}{2}v_{n}^{\intercal}\digamma+\frac{h}{4}v_{n}^{\intercal}\tilde{B}(\frac{x_{n+1}+x_{n}}{2})(v_{n+1}+v_{n})+\frac{h}{2}\digamma^{\intercal}v_{n}
+\frac{h^{2}}{2}\digamma^{\intercal}\digamma
+\frac{h^{2}}{4}\digamma^{\intercal}\tilde{B}(\frac{x_{n+1}+x_{n}}{2})(v_{n+1}+v_{n})\\
&-\frac{h}{4}(v_{n+1}+v_{n})^{\intercal}\tilde{B}(\frac{x_{n+1}+x_{n}}{2})v_{n}
-\frac{h^{2}}{8}(v_{n+1}+v_{n})^{\intercal}\tilde{B}(\frac{x_{n+1}+x_{n}}{2})\tilde{B}(\frac{x_{n+1}+x_{n}}{2})(v_{n+1}+v_{n})\\
&-\frac{h^{2}}{4}(v_{n+1}+v_{n})^{\intercal}\tilde{B}(\frac{x_{n+1}+x_{n}}{2})\digamma+U(x_{n+1})\\
=&\frac{1}{2}v_{n}^{\intercal}v_{n}+\frac{h}{2}v_{n}^{\intercal}\digamma+\frac{h}{2}v_{n}^{\intercal
}\tilde{B}(\frac{x_{n+1}+x_{n}}{2})(v_{n+1}+v_{n})+\frac{h}{2}\digamma^{\intercal}v_{n}+\frac{h^{2}}{2}\digamma^{\intercal}\digamma
+\frac{h^{2}}{2}\digamma^{\intercal}\tilde{B}(\frac{x_{n+1}+x_{n}}{2})(v_{n+1}+v_{n})\\
&-\frac{h^{2}}{8}(v_{n+1}+v_{n})^{\intercal}\tilde{B}(\frac{x_{n+1}+x_{n}}{2})\tilde{B}(\frac{x_{n+1}+x_{n}}{2})(v_{n+1}+v_{n})+U(x_{n+1}).
\end{array}
\end{equation}
On the other hand,  we have
\begin{equation}\nonumber
\begin{array}[c]{ll}
&U(x_{n})-U(x_{n+1})=-\int_{0}^{1} dU((1-\sigma)x_{n}+\sigma x_{n+1})\\
&=-\int_{0}^{1}(x_{n+1}-x_{n})^{\intercal}\nabla_{x}U((1-\sigma)x_{n}+\sigma x_{n+1})d\sigma=\digamma^{\intercal}(x_{n+1}-x_{n})\\
&=\frac{h}{2}\digamma^{\intercal}v_{n+1}+\frac{h}{2}\digamma^{\intercal}v_{n}.
\end{array}
\end{equation}
Inserting this result into \eqref{2} implies
\begin{equation}\label{3}
\begin{array}[c]{llll}
&E(x_{n+1},v_{n+1})\\
=&\frac{1}{2}v_{n}^{\intercal}v_{n}+\frac{h}{2}v_{n}^{\intercal}\digamma+\frac{h}{2}v_{n}^{\intercal}\tilde{B}(\frac{x_{n+1}+x_{n}}{2})(v_{n+1}+v_{n})+\frac{h}{2}\digamma^{\intercal}v_{n}+\frac{h^{2}}{2}\digamma^{\intercal}\digamma\\
&+\frac{h^{2}}{2}\digamma^{\intercal}\tilde{B}(\frac{x_{n+1}+x_{n}}{2})(v_{n+1}+v_{n})-\frac{h^{2}}{8}(v_{n+1}+v_{n})^{\intercal}\tilde{B}(\frac{x_{n+1}+x_{n}}{2})\tilde{B}(\frac{x_{n+1}+x_{n}}{2})(v_{n+1}+v_{n})\\
&+U(x_{n})-\frac{h}{2}\digamma^{\intercal}v_{n+1}-\frac{h}{2}\digamma^{\intercal}v_{n}\\
=&\frac{1}{2}v_{n}^{\intercal}v_{n}+U(x_{n})+\frac{h}{2}v_{n}^{\intercal}\digamma+\frac{h}{2}v_{n}^{\intercal}\tilde{B}(\frac{x_{n+1}+x_{n}}{2})(v_{n+1}+v_{n})+\frac{h^{2}}{2}\digamma^{\intercal}\digamma-\frac{h}{2}\digamma^{\intercal}v_{n+1}\\
&+\frac{h^{2}}{2}\digamma^{\intercal}\tilde{B}(\frac{x_{n+1}+x_{n}}{2})(v_{n+1}+v_{n})-\frac{h^{2}}{8}(v_{n+1}+v_{n})^{\intercal}\tilde{B}(\frac{x_{n+1}+x_{n}}{2})\tilde{B}(\frac{x_{n+1}+x_{n}}{2})(v_{n+1}+v_{n}).
\end{array}
\end{equation}
According to the second formula of \eqref{EP}, it follows that
\begin{equation}\nonumber
\begin{array}[c]{ll}\digamma^{\intercal} v_{n+1}&=\digamma^{\intercal}(v_{n}+h\digamma+\frac{h}{2}\tilde{B}(\frac{x_{n+1}+x_{n}}{2})(v_{n+1}+v_{n}))
\\&=\digamma^{\intercal}v_{n}+h\digamma^{\intercal}\digamma+\frac{h}{2}\digamma^{\intercal}\tilde{B}(\frac{x_{n+1}+x_{n}}{2})(v_{n+1}+v_{n}).\end{array}
\end{equation}
Thus, \eqref{3} becomes
\begin{equation}\label{4}
\begin{array}[c]{llll}
&E(x_{n+1},v_{n+1})\\
=&\frac{1}{2}v_{n}^{\intercal}v_{n}+U(x_{n})+\frac{h}{2}v_{n}^{\intercal}\digamma+\frac{h}{2}v_{n}^{\intercal}\tilde{B}(\frac{x_{n+1}+x_{n}}{2})(v_{n+1}+v_{n})\\
&+\frac{h^{2}}{2}\digamma^{\intercal}\digamma-\frac{h}{2}\digamma^{\intercal}v_{n}-\frac{h^{2}}{2}\digamma^{\intercal}\digamma-\frac{h^2}{4}\digamma^{\intercal}\tilde{B}(\frac{x_{n}+x_{n+1}}{2})(v_{n+1}+v_{n})\\
&+\frac{h^{2}}{2}\digamma^{\intercal}\tilde{B}(\frac{x_{n+1}+x_{n}}{2})(v_{n+1}+v_{n})-\frac{h^{2}}{8}(v_{n+1}+v_{n})^{\intercal}\tilde{B}(\frac{x_{n+1}+x_{n}}{2})\tilde{B}(\frac{x_{n+1}+x_{n}}{2})(v_{n+1}+v_{n})\\
=&\frac{1}{2}v_{n}^{\intercal}v_{n}+U(x_{n})+\frac{h}{2}(v_{n}^{\intercal}+\frac{h}{2}\digamma^{\intercal}-\frac{h}{4}(v_{n+1}+v_{n})^{\intercal}\tilde{B}(\frac{x_{n+1}+x_{n}}{2}))\tilde{B}(\frac{x_{n+1}+x_{n}}{2})(v_{n+1}+v_{n}).
\end{array}
\end{equation}
Based on the property of $\tilde{B}(x)$ and the second formula of
\eqref{EP}, we obtain
\begin{equation}\nonumber
-\frac{h}{4}(v_{n+1}+v_{n})^{\intercal}\tilde{B}(\frac{x_{n+1}+x_{n}}{2})=(\frac{h}{4}\tilde{B}(\frac{x_{n+1}+x_{n}}{2})(v_{n+1}+v_{n}))^{\intercal}=\frac{1}{2}v_{n+1}^{\intercal}-\frac{1}{2}v_{n}^{\intercal}-\frac{h}{2}\digamma^{\intercal}.
\end{equation}
Hence, \eqref{4} can be rewritten as
\begin{equation}\label{8}
\begin{array}[c]{ll}
&E(x_{n+1},v_{n+1})\\
=&\frac{1}{2}v_{n}^{\intercal}v_{n}+U(x_{n})+\frac{h}{2}(v_{n}^{\intercal}+\frac{h}{2}\digamma^{\intercal}+\frac{1}{2}v_{n+1}^{\intercal}-\frac{1}{2}v_{n}^{\intercal}-\frac{h}{2}\digamma^{\intercal})\tilde{B}(\frac{x_{n+1}+x_{n}}{2})(v_{n+1}+v_{n})\\
=&\frac{1}{2}v_{n}^{\intercal}v_{n}+U(x_{n})+\frac{h}{4}(v_{n+1}+v_{n})^{\intercal}\tilde{B}(\frac{x_{n+1}+x_{n}}{2})(v_{n+1}+v_{n}).
\end{array}
\end{equation}
Then,  it can be checked that
\begin{equation}\nonumber
\begin{array}[c]{ll}
(\frac{h}{4}(v_{n+1}+v_{n})^{\intercal}\tilde{B}(\frac{x_{n+1}+x_{n}}{2})(v_{n+1}+v_{n}))^{\intercal}=-\frac{h}{4}(v_{n+1}+v_{n})^{\intercal}\tilde{B}(\frac{x_{n+1}+x_{n}}{2})(v_{n+1}+v_{n}),\\
\end{array}
\end{equation}
which shows that
$\frac{h}{4}(v_{n+1}+v_{n})^{\intercal}\tilde{B}(\frac{x_{n+1}+x_{n}}{2})(v_{n+1}+v_{n})=0$.
Therefore, \eqref{8} becomes
$$E(x_{n+1},v_{n+1})=\frac{1}{2}v_{n}^{\intercal}v_{n}+U(x_{n})=E(x_{n},v_{n}).$$
This implies the statement of the theorem.
\end{proof}

We note that the integral in \eqref{EP} can be evaluated exactly if
$ U $ is a special function, we have the following theorem.
\begin{mytheo}
Assume $ U=U(a^{\intercal}x) $ with $ a\in \mathbb{R}^{3} $, then
$$ \int_{0}^{1}F(x_{n}+\tau(x_{n+1}-x_{n}))d\tau=\frac{-a}{a^{\intercal}x_{n+1}-a^{\intercal}x_{n}}(U(a^{\intercal}x_{n+1})-U(a^{\intercal}x_{n})).$$
\end{mytheo}
\begin{proof}
This result is obtained immediately by considering the following
fact
\begin{equation}\nonumber
\begin{array}
[c]{ll} &\int_{0}^{1}F(x_{n}+\tau(x_{n+1}-x_{n}))d\tau\\
=&\int_{0}^{1}F((1-\tau)x_{n}+\tau
x_{n+1})d\tau\\
=&-\int_{0}^{1}aU^{'}(a^{\intercal}((1-\tau)x_{n}
+\tau x_{n+1}))d\tau\\
=&\frac{-a}{a^{\intercal}x_{n+1}-a^{\intercal}x_{n}}\int_{0}^{1}\frac{dU(a^{\intercal}((1-\tau)x_{n}
+\tau
x_{n+1}))}{d\tau}d\tau\\
=&\frac{-a}{a^{\intercal}x_{n+1}-a^{\intercal}x_{n}}(U(a^{\intercal}x_{n+1})-U(a^{\intercal}x_{n})).
\end{array}
\end{equation}

\end{proof}

 We next consider the case that the integral in \eqref{EP} cannot be evaluated exactly. Under this situation, it is natural to use
  a numerical quadrature formula for the integral. Here we choose $s$-point  Gauss-Legendre's quadrature for the integral in \eqref{EP} and
 get the following result.

\begin{mytheo}\label{thm ep2}
 Assume that $U(x)$ is  a polynomial in $ x $ of degree $ n .$
  Let  $(b_{i},c_{i})$ for $i=1,\dots,s $ respectively  be the weights and the nodes of an $s$-point Gauss-Legendre's quadrature on $[0,1]$ that is exact for polynomials of degree $\leq n-1 $.
 Then the following modified energy-preserving method
\begin{equation}\label{EPGL}
\begin{array}
[c]{ll}
&x_{n+1}=x_{n}+hv_{n}+\frac{h^2}{2}\sum\limits_{i=1}^{s}b_{i}F(x_{n}+c_{i}(x_{n+1}-x_{n}))+\frac{h^2}{2}\tilde{B}(\frac{x_{n+1}+x_{n}}{2})\frac{v_{n+1}+v_{n}}{2},\\
&v_{n+1}=v_{n}+h\sum\limits_{i=1}^{s}b_{i}F(x_{n}+c_{i}(x_{n+1}-x_{n}))+h\tilde{B}(\frac{x_{n+1}+x_{n}}{2})\frac{v_{n+1}+v_{n}}{2}.
\end{array}
\end{equation}
exactly preserves $ E $ defined by \eqref{E(x,v)}.
\end{mytheo}

\section{Long-time momentum conservation}\label{sec4}
 In this section, we study the long-time momentum
conservation of the  method \eqref{EPGL}.
\subsection{Main result of the section}
\begin{mytheo}\label{Thm lm}
If the numerical solution  \eqref{EPGL} stays in a compact set that
is independent of $ h $ and  $B(x)=B $ is a constant magnetic field,
then for arbitrary positive integers $ N $, the method \eqref{EPGL}
conserves the momentum over long times as follows
\begin{equation}\nonumber
M(x_{n},v_{n})=M(x_{0},v_{0})+\mathcal{O}(h^{2})  \ \ \ \
\textmd{for} \ \ \ \ 0 \leq nh \leq Ch^{-N+2},
\end{equation}
where the constant symbolized by $\mathcal{O}$ is independent of $ n
$ and $ h $.
\end{mytheo}
\begin{rem}
According to this result, it is known that the order of the
Gauss-Legendre's quadrature using in  \eqref{EPGL} does not
influence the long time  momentum  conservation. This point can be
seen from the numerical results given in Section \ref{sec6}.
\end{rem}

\subsection{Proof of Theorem \ref{Thm lm}}
In order to prove the result, we need to use backward error analysis
(see Chap. IX of \cite{hairer2006}).

To this end, we require a modified differential equation and its
solution $ (x(t),v(t)) $ satisfies $ (x(nh),v(nh))=(x_{n},v_{n} )$
with the solution  $ (x_{n},v_{n}) $ obtained by the  method
 \eqref{EPGL}. Such a function has to satisfy
\begin{equation}\label{7}
\begin{array}
[c]{ll}
&x(t+h)=x(t)+\frac{h}{2}(v(t+h)+v(t)),\\
&v(t+h)=v(t)+h\tilde{B}(\frac{x(t+h)+x(t)}{2})\frac{v(t+h)+v(t)}{2}+h\sum\limits_{i=1}^{s}b_{i}F(x(t)+c_{i}(x(t+h)-x(t))).
\end{array}
\end{equation}
Based on those formulas, we define
$$L_{1}(\varphi)=\varphi-1,L_{2}(\varphi)=\frac{\varphi+1}{2}.$$
Then \eqref{7} becomes
\begin{equation}\nonumber
\begin{array}
[c]{ll}
&L_{1}(e^{hD})x(t)=hL_{2}(e^{hD})v(t),\\
&L_{1}(e^{hD})v(t)=h\tilde{B}(L_{2}(e^{hD})x(t))L_{2}(e^{hD})v(t)+h\sum\limits_{i=1}^{s}b_{i}F(x(t)+c_{i}L_{1}(e^{hD})x(t)),
\end{array}
\end{equation}
where $D$ is the differential operator (see \cite{hairer2006}).
 By letting $z(t)=L_{2}(e^{hD})x(t),$ we have
\begin{equation}\label{L}
\begin{array}
[c]{ll}
\frac{1}{h^{2}}L_{1}^{2}L_{2}^{-2}(e^{hD})z(t)=\frac{1}{h}\tilde{B}(z(t))L_{1}L_{2}^{-1}(e^{hD})z(t)
+\sum\limits_{i=1}^{s}b_{i}F(L_{2}^{-1}(e^{hD})z(t)+c_{i}L_{1}L_{2}^{-1}(e^{hD})z(t)).
\end{array}
\end{equation}
Based on the following properties
\begin{equation}\nonumber
\begin{array}
[c]{ll}
&L_{1}^{2}L_{2}^{-2}(e^{hD})=h^{2}D^{2}-\frac{1}{6}h^{4}D^{4}+\frac{17}{720}h^{6}D^{6}+ \ldots,\\
&L_{1}L_{2}^{-1}(e^{hD})=hD-\frac{1}{12}h^{3}D^{3}+\frac{1}{120}h^{5}D^{5}- \ldots,\\
&L_{2}^{-1}(e^{hD})=1-\frac{1}{2}hD+\frac{1}{24}h^{3}D^{3}-\frac{1}{240}h^{5}D^{5}+
\ldots,
\end{array}
\end{equation}
the formula \eqref{L} can rewritten as
\begin{equation}\label{Me}
\begin{array}
[c]{ll} \ddot{z}-\frac{1}{6}h^{2}z^{(4)}+\frac{17}{720}h^{4}z^{(6)}-
\ldots=\tilde{B}(z)(\dot{z}-\frac{1}{12}h^{2}z^{(3)}+\frac{1}{120}h^{4}z^{(5)}-
\ldots)+\sum\limits_{i=1}^{s}b_{i}F(\omega),
\end{array}
\end{equation}
where $
\omega=z+h(c_{i}-\frac{1}{2})\dot{z}+\frac{1}{12}h^{3}(-c_{i}+\frac{1}{2})z^{(3)}+\frac{1}{120}h^{5}(c_{i}-\frac{1}{2})z^{(5)}+
\ldots $.

\begin{lem}\label{theo}
If $ U(x) $ and $ A(x) $  satisfy the properties \eqref{M}, then
there exist $ h $-independent functions $ M_{2i}(x,v)$ such that the
function
$$ M_{h}(x,v)=M(x,v)+h^{2}M_{2}(x,v)+h^{4}M_{4}(x,v)+\ldots ,$$
truncated at the $\mathcal{O}(h^{N})$ term, satisfies
\begin{equation}\nonumber
\frac{d}{dt}M_{h}(z,\dot{z})=h^{2}z^{\intercal}S\tilde{B}(z)(-\frac{1}{12}z^{(3)}+\frac{h^{2}}{120}z^{(5)}-
\ldots )+\mathcal{O}(h^{N})
\end{equation}
along solutions of the modified differential equation \eqref{Me}.
\end{lem}
\begin{proof}
It is noted that $ z^{\intercal}Sz^{(k)} $ can be written as a total
differential for even values of $ k $, we multiply \eqref{Me} with $
z^{T} S $ yields
\begin{equation}\nonumber
\begin{array}
[c]{ll}
z^{\intercal}S(\ddot{z}-\frac{1}{6}h^{2}z^{(4)}-\dot{z}\times
B(z)+\ldots)
=z^{\intercal}S\tilde{B}(z)(-\frac{h^{2}}{12}z^{(3)}+\frac{h^{4}}{120}z^{(5)}-\ldots)-z^{\intercal}S\sum\limits_{i=1}^{s}\nabla
U(\omega),
\end{array}
\end{equation}
i.e.
\begin{equation}\nonumber
\begin{array}
[c]{ll}
\frac{d}{dt}(z^{\intercal}S\dot{z}-\frac{h^{2}}{6}(z^{\intercal}Sz^{(3)}-\dot{z}^{\intercal}S\ddot{z})+z^{\intercal}SA(z)+
\ldots)
=h^{2}z^{\intercal}S\tilde{B}(z)(-\frac{1}{12}z^{(3)}+\frac{h^{2}}{120}z^{(5)}-\ldots).
\end{array}
\end{equation}
where we used the facts   that $z^{\intercal}S\nabla U(\omega)=0 $
and $z^{\intercal}S(\dot{z}\times
B(z))=-\frac{d}{dt}(z^{\intercal}SA(z))$. This result shows the
statement of this lemma immediately.

\end{proof}

When $ B(x)=B$, it is shown in   \cite{Hairer2017} that the
properties \eqref{M} are satisfied if $ S $ is the skew-symmetric
matrix that embodies the cross product with $ B $, i.e. $ Sv=v\times
B $, and $ U $ is invariant under rotations with the axis $ B $,
i.e. $ \nabla U(x)\times B=0 $ for all $ x $. Under those
conditions, the above result can be improved as follows.

\begin{cor}\label{cor1}
If $ B(x)=B $ is a constant magnetic field and $ \nabla U(x)\times
B=0 $ for all $ x $, then there exist $h$-independent functions $
\widetilde{M}_{2i}(x,v) $, such that the function
$$\widetilde{M}_{h}(x,v)=M(x,v)+h^{2}\widetilde{M}_{2}(x,v)+h^{4}\widetilde{M}_{4}(x,v)+\ldots ,$$
truncated at the $ \mathcal{O}(h^{N}) $ term, satisfies
\begin{equation}\nonumber
\frac{d}{dt}\widetilde{M}_{h}(z,\dot{z})=\mathcal{O}(h^{N})
\end{equation}
along solutions of the modified differential equation \eqref{Me}.
\end{cor}

\begin{proof}
From the proof of Lemma \ref{theo}, it follows that
\begin{equation}\nonumber
\begin{array}
[c]{ll}
&\frac{d}{dt}(z^{\intercal}S\dot{z}-\frac{h^{2}}{6}(z^{\intercal}Sz^{(3)}-\dot{z}^{\intercal}S\ddot{z})+z^{\intercal}SA(z)+\ldots)\\
&=h^{2}z^{\intercal}S\tilde{B}(z)(-\frac{1}{12}z^{(3)}+\frac{h^{2}}{120}z^{(5)}-\ldots)+\mathcal{O}(h^{N}).
\end{array}
\end{equation}
Since $ Sz=\tilde{B}z $ , then we obtain
\begin{equation}\nonumber
\begin{array}
[c]{lll}
&\frac{d}{dt}(z^{\intercal}S\dot{z}-\frac{h^{2}}{6}(z^{\intercal}Sz^{(3)}-\dot{z}^{\intercal}S\ddot{z})+z^{\intercal}SA(z)+\ldots)\\
&=-h^{2}(\tilde{B}z)^{\intercal}\tilde{B}(-\frac{1}{12}z^{(3)}+\frac{h^{2}}{120}z^{(5)}-\ldots)+\mathcal{O}(h^{N}).
\end{array}
\end{equation}
On the other hand, $ (\tilde{B}z)^{\intercal}(\tilde{B}z^{(k)}) $
can be written as a total differential for odd values of $ k $ as
follows
\begin{equation}\nonumber
\begin{array}
[c]{ll}
&-h^{2}(\tilde{B}z)^{\intercal}\tilde{B}(-\frac{1}{12}z^{(3)}+\frac{h^{2}}{120}z^{(5)}-\ldots)\\
&=-h^{2}\frac{d}{dt}(-\frac{1}{12}((\tilde{B}z)^{\intercal}(\tilde{B}z)^{(2)}-\frac{1}{2}(\tilde{B}\dot{z})^{\intercal}(\tilde{B}\dot{z}))
+\frac{h^{2}}{120}((\tilde{B}z)^{\intercal}(\tilde{B}z)^{(4)}-(\tilde{B}\dot{z})^{\intercal}(\tilde{B}z)^{(3)}+\frac{1}{2}(\tilde{B}\ddot{z})^{\intercal}(\tilde{B}\ddot{z})-\ldots)).
\end{array}
\end{equation}
Hence, the conclusion  is proved.
\end{proof}

By a standard argument given in Chap. IX of \cite{hairer2006},
Theorem \ref{Thm lm} is proved immediately by considering Corollary
\ref{cor1}.

\section{Numerical experiment}\label{sec6}

In this  section, we carry out a  numerical experiment to show the
efficiency of our EP methods. The   methods for comparison  are
chosen as follows:
\begin{itemize}
  \item BORIS: the Boris method presented in \cite{Boris};
  \item EP1: the one-stage EP  method \eqref{EPGL} presented in this paper using the  one-point  Gauss-Legendre's
  quadrature;
  \item EP2: the two-stage EP method \eqref{EPGL} presented in this paper using the two-point  Gauss-Legendre's
   quadrature;
  \item EP3: the three-stage EP method \eqref{EPGL} presented in this paper using the three-point  Gauss-Legendre's
    quadrature.
\end{itemize}

For the charged-particle dynamics \eqref{charged-particle}, we
consider potential $$U(x)=\frac{1}{100\sqrt{x_{1}^2+x_{2}^2}},$$ and
the filed $$B(x)=(0,0,\sqrt{x_{1}^{2}+x_{2}^{2}})^{\intercal}.$$ The
initial values are chosen as $x(0)=(0.0,1.0,0.1)^{\intercal}$ and
$v(0)=(0.09,0.05,0.20)^{\intercal}$. We solve the problem in the
interval $ [0,T] $ with different stepsizes $h=\frac{1}{2^{i}}$ for
$ i=6,7,8,9 $. The global errors are presented in Figure
\ref{fig:err} for $T=10,100,1000.$  We then integrate this problem
with the stepsizes $ h=0.05 $ and $ h=0.1 $ in the integral
[0,10000]. See Figure \ref{fig:h} for the energy conservation for
different methods. Besides the energy we also consider the momentum
$$ M(x,v)=(v_{1}+A_{1}(x))x_{2}-(v_{2}+A_{2}(x))x_{1}.$$ Its errors
are presented in Figure \ref{fig:M}.

From the results, it can be clearly observed that our methods
provide a better numerical solution than Boris method and preserve
the energy and the momentum well. Moreover, the energy conservation
is much better than the Boris method and the momentum conservation
is unchanged no matter which Gauss-Legendre's
  quadrature is used. These observed long time conservations support the theoretical results given in Theorems \ref{thm ep1}, \ref{thm ep2} and \ref{Thm lm}.

\begin{figure}[ptbh]%
\centering\tabcolsep=1mm
\begin{tabular}
[c]{cccc}%
\includegraphics[width=4.3cm,height=6cm]{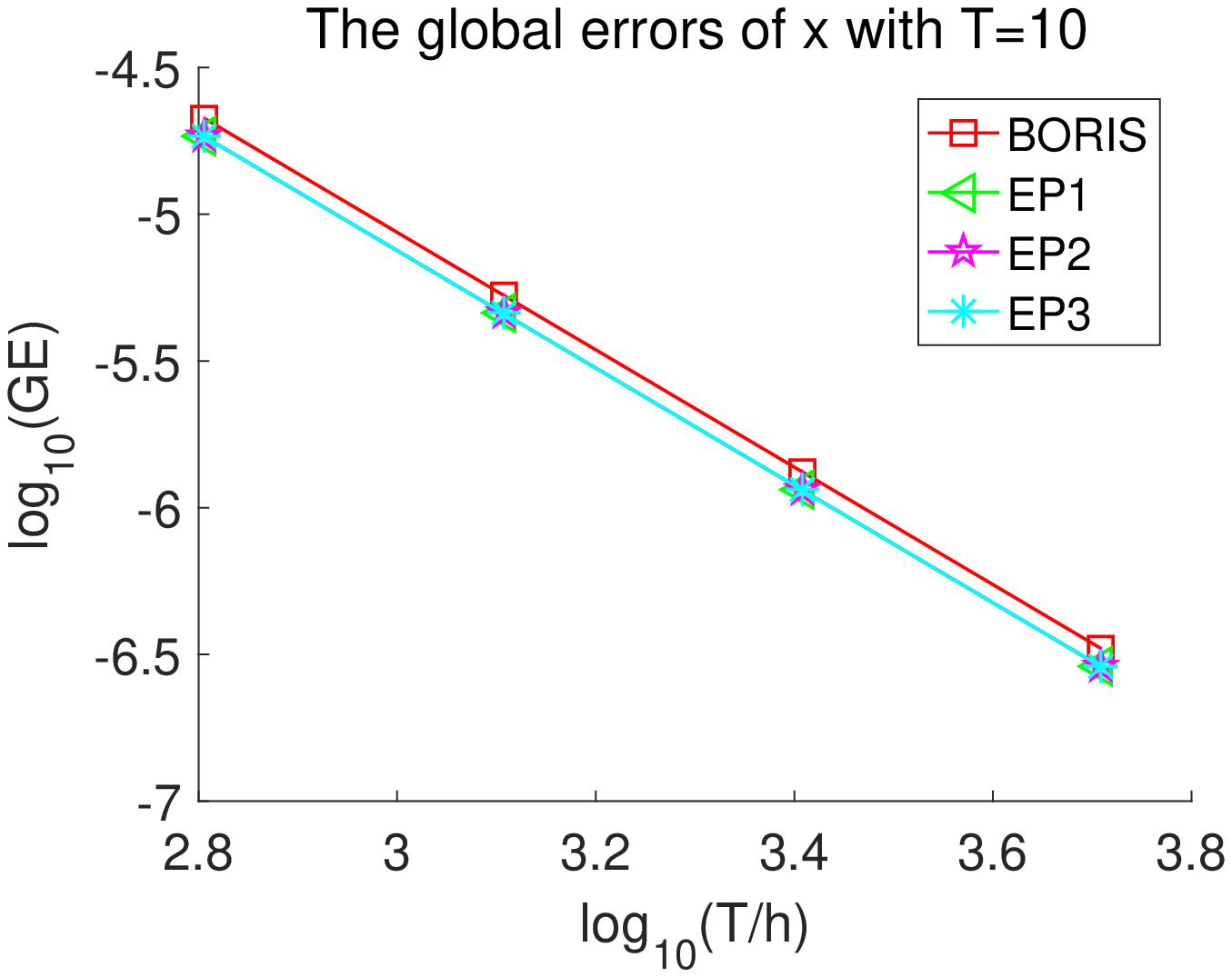} & \includegraphics[width=4.3cm,height=6cm]{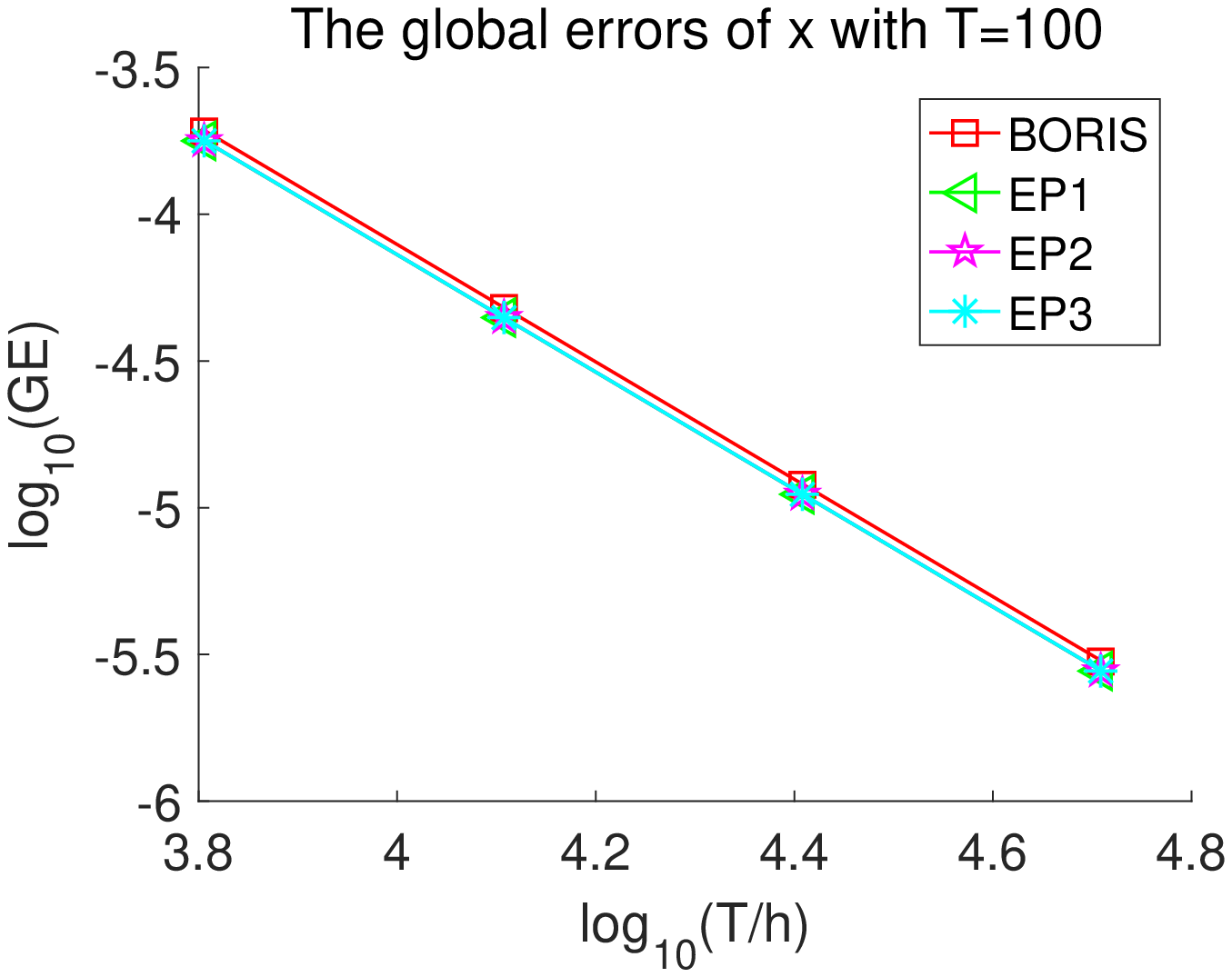} &\includegraphics[width=4.3cm,height=6cm]{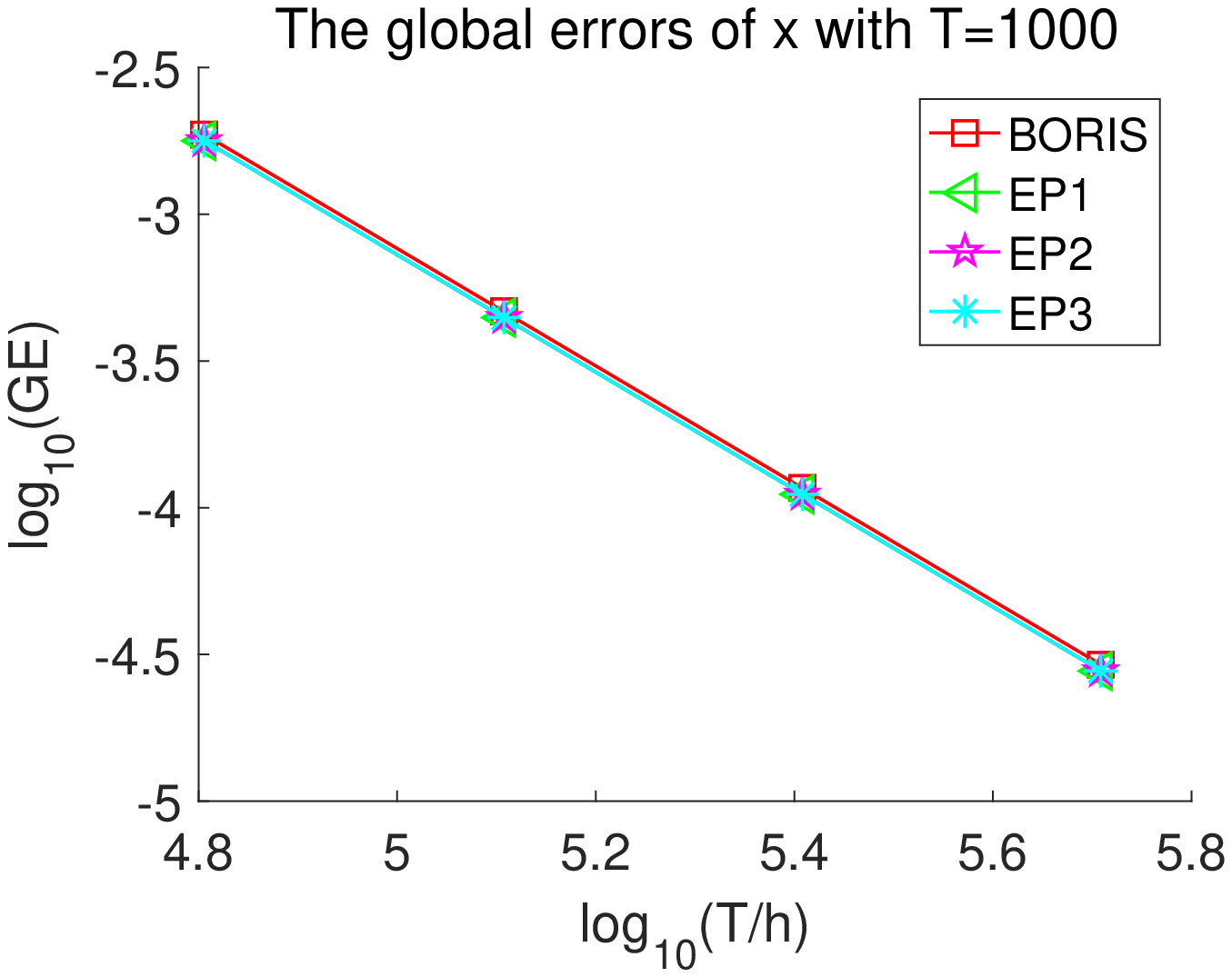}\\
\includegraphics[width=4.3cm,height=6cm]{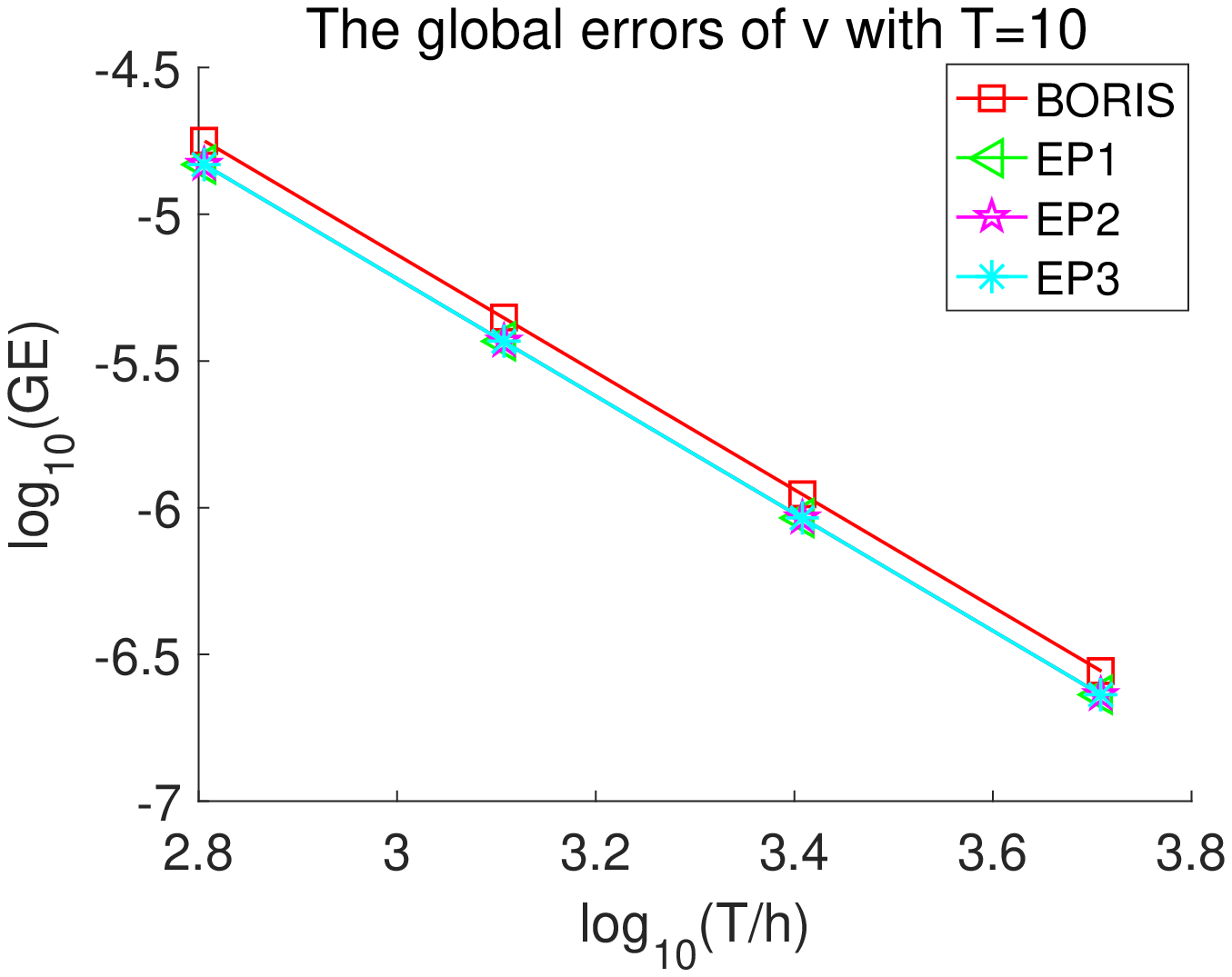} & \includegraphics[width=4.3cm,height=6cm]{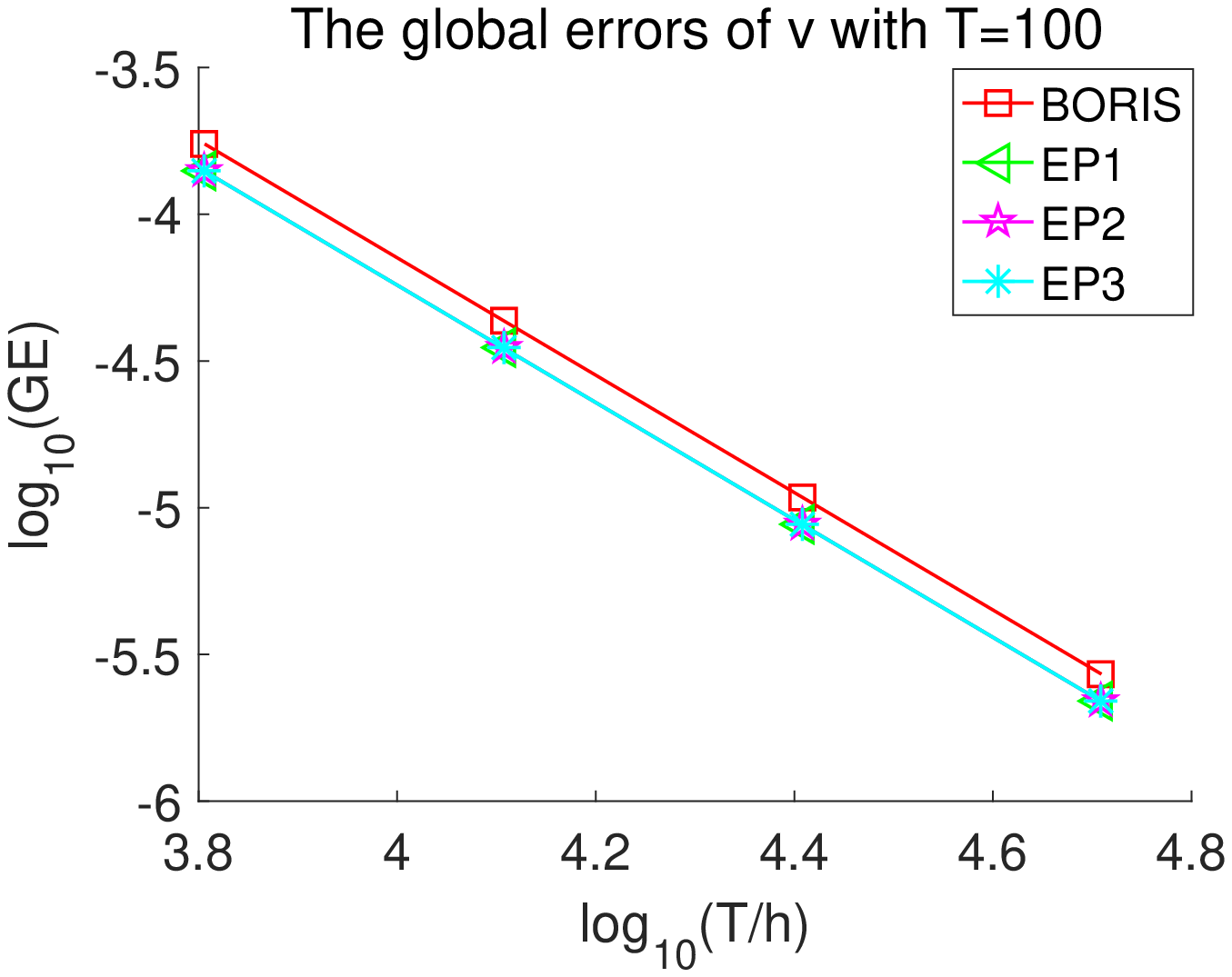} &\includegraphics[width=4.3cm,height=6cm]{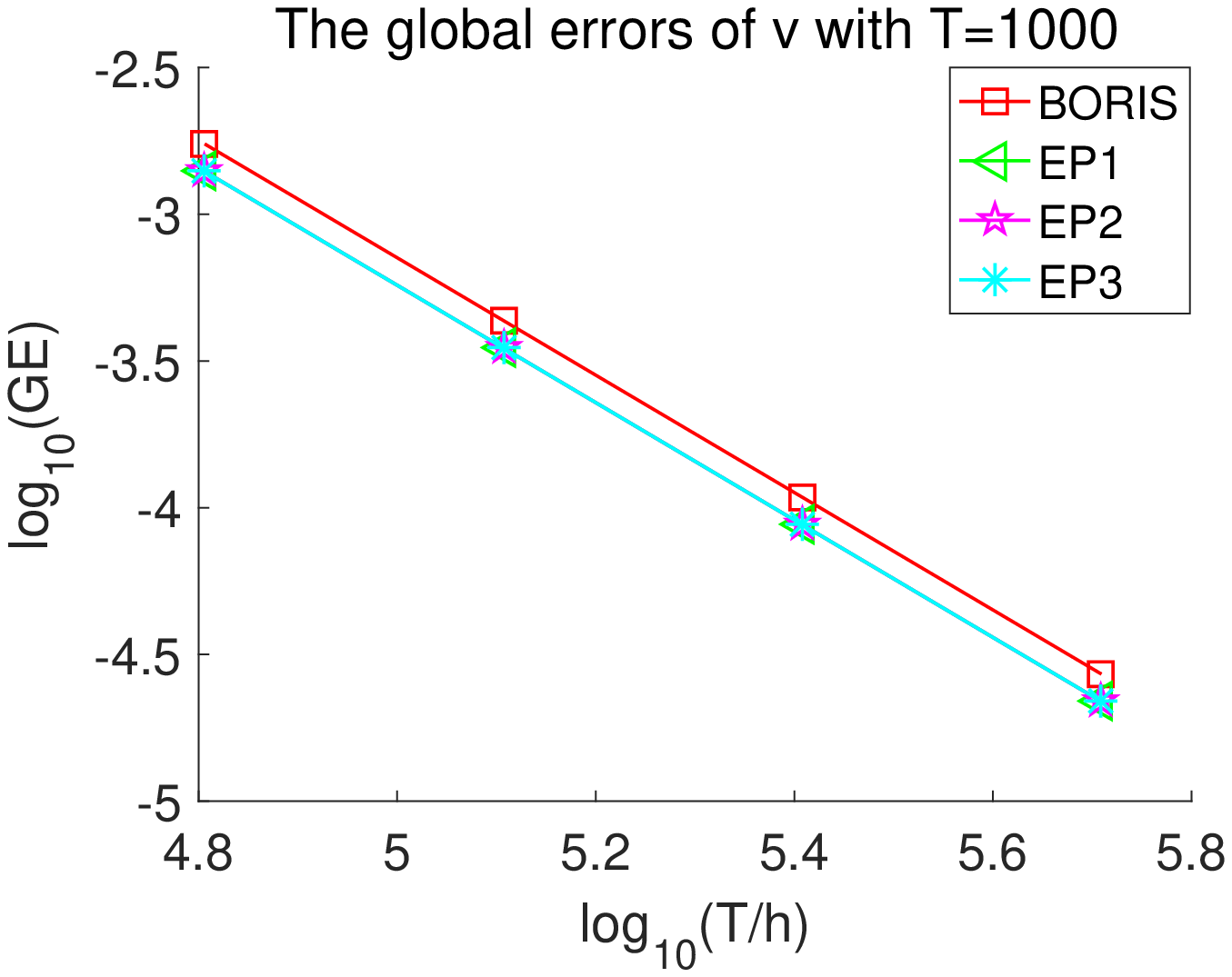} \\ \
\end{tabular}
\caption{The logarithm of the global error against the logarithm of T/h.}%
\label{fig:err}
\end{figure}

\begin{figure}[ptbh]%
\centering\tabcolsep=1mm
\begin{tabular}
[c]{cccc}%
&\includegraphics[width=12cm,height=4cm]{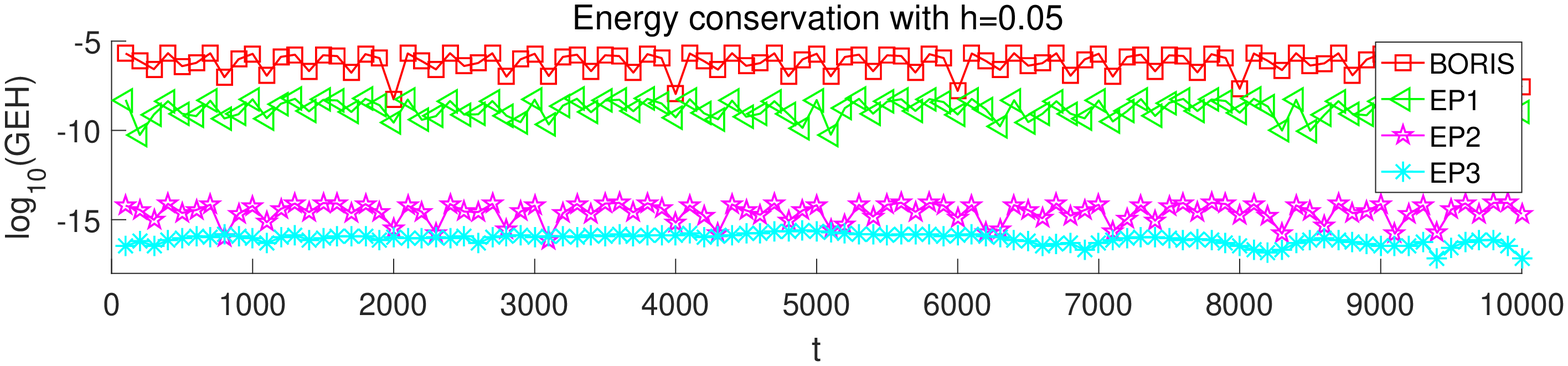} \\
 &\includegraphics[width=12cm,height=4cm]{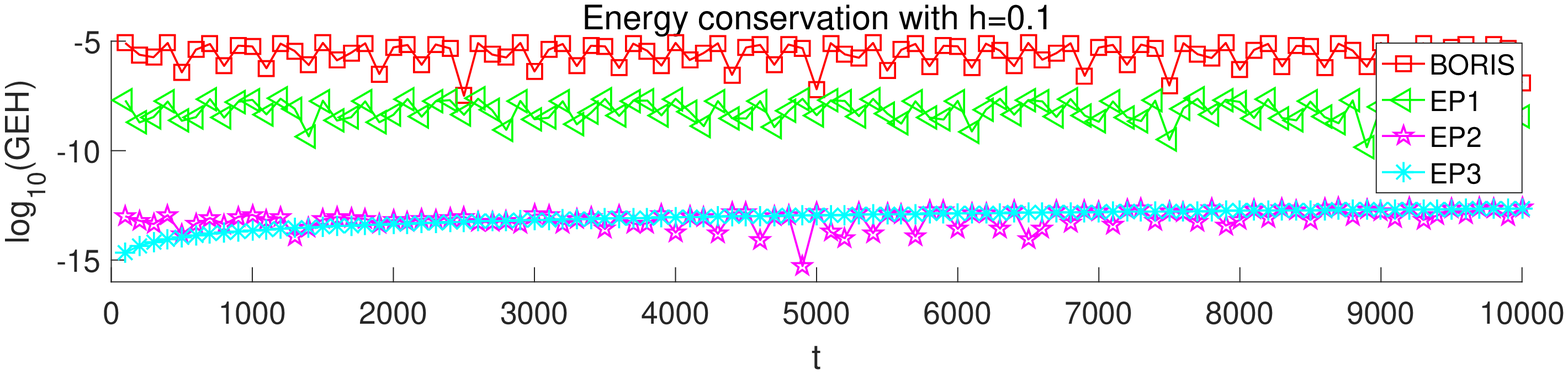}
\end{tabular}
\caption{The logarithm of the error of energy against t.}%
\label{fig:h}
\end{figure}

\begin{figure}[ptbh]%
\centering\tabcolsep=1mm
\begin{tabular}
[c]{cccc}%
&\includegraphics[width=12cm,height=4cm]{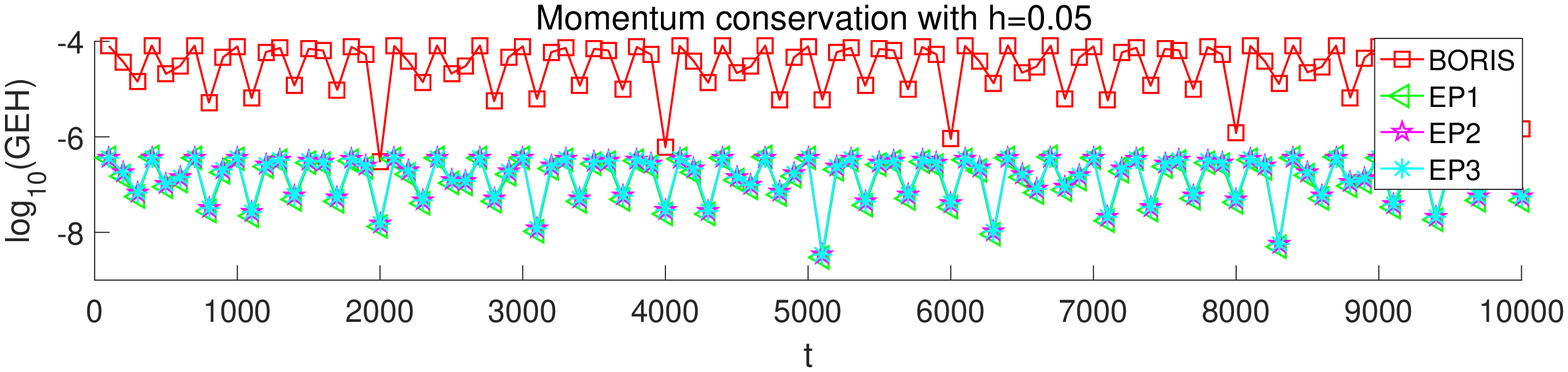} \\
 &\includegraphics[width=12cm,height=4cm]{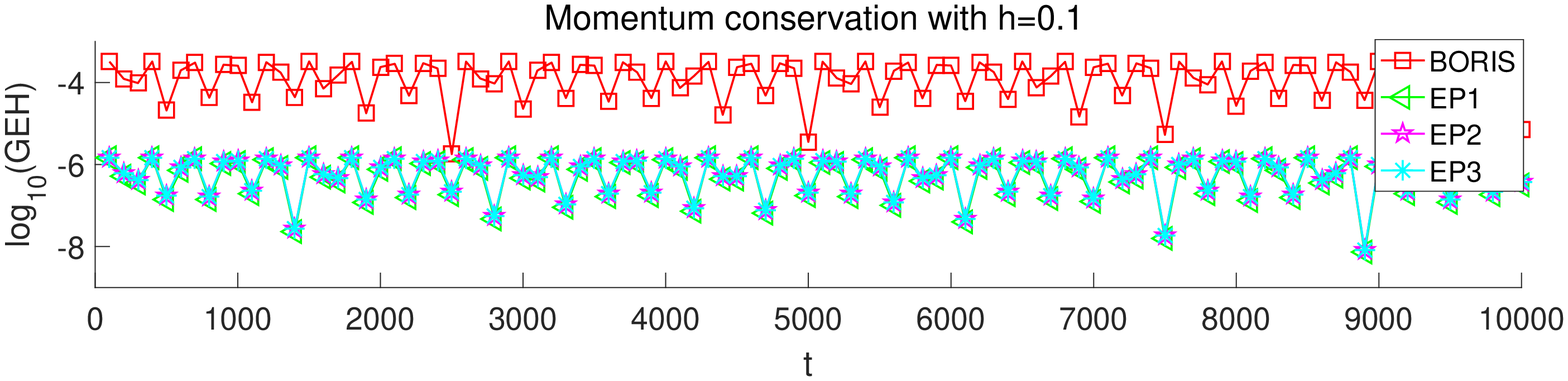}
\end{tabular}
\caption{The logarithm of the error of momentum against t.}%
\label{fig:M}
\end{figure}

\section{Conclusion}\label{sec7}
In this paper,  the energy-preserving methods for solving
charged-particle dynamics \eqref{charged-particle} were presented
and studied. We analyzed and discussed its algebraic order and
symmetry. Moreover, it was shown that our method can exactly
preserve the energy of the charged-particle dynamics. We also proved
that the momentum is nearly conserved along the novel methods over
long times. A numerical experiment was performed and it was shown
that our method is more effective and it can preserve the energy and
momentum better than the Boris method.

\end{document}